\documentclass[12pt]{article}
\usepackage{amsmath,amssymb, amsfonts,amsthm,fancyhdr,mathrsfs,geometry,graphicx,setspace,float, eufrak}
\usepackage{enumerate}

\title{\textbf{Ideals of mid $\boldsymbol{p}$-summing operators: a tensor product approach}}

\author{DEEPIKA BAWEJA, ALEENA PHILIP }
\date{\small\textit{Department of Mathematics, BITS Pilani Hyderabad, 500078\\%
deepika@hyderabad.bits-pilani.ac.in, p20190038@hyderabad.bits-pilani.ac.in}}
\newtheorem{thm}{Theorem}[section]
\newtheorem{pro}[thm]{Proposition}
\newtheorem{cor}[thm]{Corollary}

\newtheorem{Definition}[thm]{Definition}
\newtheorem{Lemma}[thm]{Lemma}

\numberwithin{equation}{section}

\begin{document}
	\maketitle	
	\begin{abstract}
	In this article, we study the ideals of mid $p$-summing operators. We obtain  representation of these operator ideals by tensor norms. These tensor norms are defined by using  a particular kind of sequential dual of the class of mid $p$-summable sequences. As a consequence, we prove a characterization of the adjoints of weakly and absolutely mid $p$-summing operators in terms of the operators that are defined by the transformation of dual spaces of certain vector-valued sequence spaces. 	\end{abstract}

	\\ MSC Code 2020: 46B28, 46B45, 47B10.
	
		\section{Introduction and Terminologies}

		The theory of mid $p$-summing operators has its origin in the work of Karn and Sinha\cite{Karn} with the introduction  of mid $p$-summable sequences. The concept of mid $p$-summability   appear naturally while extending the notion of limited sets to the $p$-sense and, lies intermediate to the notion of weak $p$-summability and absolute $p$-summability.
	This new kind of summability for vector-valued sequences has provoked the interest of several authors (see
		 \cite{Botelho,Fourie,zee}) which led them to investigate the ideals of  operators that transform sequences from/into the space of mid $p$-summable sequences. Most of these studies involve only
		  the theory of operator ideals. 
		  
		  On the other hand, the theory of tensor products of Banach spaces has been initiated in Grothendieck's famous R\'esum\'e\cite{Gro}. The theory of operator ideals is closely connected with the theory of tensor norms (see \cite{Defant}); indeed, representation theorem for maximal ideals\cite[Theorem 17.5]{Defant} provides a natural correspondence between maximal operator ideals and finitely generated tensor norms: there exists a finitely generated tensor norm $\alpha$ associated with a maximal operator ideal $[\mathcal{U},\|\cdot\|_{\mathcal{U}}]$ such that $\mathcal{U}(X,Y^{*})=(X\otimes_{\alpha}Y)^{*}$ holds isometrically for all Banach spaces $X,Y$. Recently, Zeekoi\cite{zee} has proved that the class of weakly mid $p$-summing operators is a maximal Banach operator ideal which ensures that one can construct a finitely generated tensor norm associated to the ideal of weakly mid $p$-summing operators. The aim of this paper is to define a tensor norm corresponding to the ideal of weakly mid $p$-summing operators and further explore its properties using the theory of tensor products.

	Our notation of operator ideals and tensor norms is standard and we refer to the classical monographs \cite{Defant, DiJaTo, ryan} for a detailed background. We shall use the letters $X$, $Y$ to denote Banach spaces over $\mathbb{K}$ and $B_{X}$ to denote the closed unit ball of $X$. We denote by $\mathcal{L}(X,Y)$  the space of all bounded linear operators from $X$ to $Y$, and  the topological dual  and bidual of $X$ by $X^{*}$ and  $X^{**}$  respectively.  The symbol $X\otimes Y$  denotes the tensor product of Banach spaces $X$ and $Y$.  A typical element in $X\otimes Y$ is of the form $u=\sum\limits_{j=1}^{n}x_{j}\otimes y_{j}$, where $x_{j} 's \in X$ and $y_{j} 's \in Y$. A norm $\alpha$ on $X\otimes Y$ is said to be a \textit{reasonable cross norm} if the following inequality holds for every $u\in X\otimes Y$:
		\begin{equation}\label{eq:1}
			\varepsilon\left(u\right) \leq \alpha\left(u\right) \leq \pi\left(u\right)
		\end{equation}
		where $\varepsilon$ and $\pi$ denote injective and projective norms respectively. A reasonable crossnorm $\alpha$ is  said to be \textit{uniform} if for any Banach spaces $X_{1},X_{2},Y_{1},Y_{2}$ and bounded linear maps $S_{1}:X_{1}\rightarrow Y_{1}$ and $S_{2}:X_{2}\rightarrow Y_{2}$, $$\|S\otimes T:X_{1}\otimes_{\alpha} X_{2}\rightarrow Y_{1}\otimes_{\alpha} Y_{2}\|\leq \|S\|\|T\|,$$ where $S_{1}\otimes_{\alpha}S_{2}$ is defined as $S_{1}\otimes_{\alpha}S_{2}(x_{1}\otimes_{\alpha}x_{2})=S_{1}(x_{1})\otimes_{\alpha}S_{2}(x_{2}), x_{1}\in X_{1}, x_{2}\in X_{2}$. A \textit{tensor norm} $\alpha$ is a uniform reasonable crossnorm. Moreover, a uniform crossnorm is said to be \textit{finitely generated} if for every pair of Banach spaces $X,Y$ and for each $u\in X \otimes Y$,
		\begin{equation}\label{eq:2}
			\alpha_{X,Y}\left(u\right) = \inf \{\alpha_{M,N}\left(u\right) : u\in M\otimes N, \text{dim}\left(M\right), \text{dim}\left(N\right) <\infty\}.
		\end{equation}
A sequence class $S$  is a rule that assigns to each Banach space $X,$  a Banach space $S(X)$ of $X$-valued sequences, such that:
\begin{enumerate}[(i)]
	\item $c_{00}(X)\subseteq S(X) \hookrightarrow \ell_{\infty}(X)$ for each Banach space $X$.
	\item $\|x.e_{j}\|_{S(X)}
	=\|x\|_{X}	$ for each $x\in X$.	
\end{enumerate}
		
In the sequel, we deal with the the following vector-valued sequence classes:
		\\Let  $p^{*}$ be the conjugate of $p$, i.e.$\frac{1}{p}+\frac{1}{p^{*}}=1$. Then
		\begin{itemize}
			\item For $1\leq p < \infty$, $\ell_{p}\left(X\right):=\left\{(x_{j})_{j=1}^{\infty}\in X^{\mathbb{N}}:\sum\limits_{j=1}^{\infty}\|x_{j}\|^{p} < \infty \right\}$ is the Banach space of \textit{absolutely $p$-summable sequences} endowed with the norm,
			\begin{equation*}
				\|(x_{j})_{j=1}^{\infty}\|_{p}=\left(\sum_{j=1}^{\infty}\|x_{j}\|^{p}\right)^{1/p}.
			\end{equation*}
		For $p=\infty,$ $\ell_{\infty}(X)=\left\{(x_{j})_{j=1}^{\infty}\in X^{\mathbb{N}}: \|x_{j}\|<\infty \text{ for every } j\right\} $ is the space of all bounded sequences in $X$ endowed with the supremum norm,
		\begin{equation*}
			\|(x_{j})_{j=1}^{\infty}\|_{\infty}=\sup_{j}\|x_{j}\|.
		\end{equation*}	
		\item  For $1\leq p < \infty$, $\ell_{p}^{w}\left(X\right):=\left\{(x_{j})_{j=1}^{\infty}\in X^{\mathbb{N}}:\sum\limits_{j=1}^{\infty}|x^{*}\left(x_{j}\right)|^{p} < \infty \textrm{ for each }   x^{*} \in X^{*}\right\}$ is the Banach space of \textit{ weakly $p$-summable sequences} endowed with the norm,
			\begin{equation*}
				\|(x_{j})_{j=1}^{\infty}\|_{p}^{w}=\sup_{x^{*}\in B_{X^{*}}}\left(\sum_{j=1}^{\infty}|x^{*}\left(x_{j}\right)|^{p}\right)^{1/p}.
			\end{equation*}
	Note that for $p=\infty, ~ \ell_{\infty}^{w}(X)=\ell_{\infty}(X).$
		\item  For $1\leq p \leq\infty$,
			$\ell_{p}\left<X\right>:=\left\{(x_{j})_{j=1}^{\infty}\in X^{\mathbb{N}}:\sum\limits_{j=1}^{\infty}|x_{j}^{*}\left(x_{j}\right)| < \infty \textrm{ for each } (x_{j}^{*})_{j=1}^{\infty} \in \ell_{p^{*}}^{w}(X^{*})\right\}$ is the Banach space of  \textit{Cohen strongly $p$-summable sequences} endowed with the norm,
			\begin{equation*}
				\|(x_{j})_{j=1}^{\infty}\|_{C,p}=\sup_{(x_{j}^{*})_{j=1} ^{\infty}\in B_{\ell_{p^{*}}^{w}(X^{*})}}\sum_{j=1}^{\infty}|x_{j}^{*}\left(x_{j}\right)|.
			\end{equation*}
	
			\item  For $1\leq p < \infty$,
			$\ell_p^{mid}(X):=\left\{(x_{j})_{j=1}^{\infty}\in X^{\mathbb{N}}: \sum\limits_{n=1}^{\infty}\sum\limits_{j=1}^{\infty}|x_{n}^{*}\left(x_{j}\right)|^{p} < \infty \textrm{ for each } (x_{n}^{*})_{n=1}^{\infty} \in \ell_{p}^{w}(X^{*})\right\}$ is the Banach space of \textit{mid $p$-summable sequences} endowed with the norm,
			\begin{equation*}
				\|(x_{j})_{j=1}^{\infty}\|_{p}^{mid}=\sup_{(x^{*}_{n})_{n=1}^{\infty}\in B_{\ell_{p}^{w}(X^{*})}}\left(\sum_{n=1}^{\infty}\sum_{j=1}^{\infty}|x_{n}^{*}\left(x_{j}\right)|^{p}\right)^{1/p}.
			\end{equation*}
		
		\end{itemize}
		Clearly, the following inclusion holds for each  $1\leq p \leq\infty$:
		\begin{equation}\label{eq:3}
			\ell_{p}\left<X\right>\subsetneq \ell_{p}\left(X\right)\subsetneq \ell_p^{mid}(X)\subsetneq \ell_{p}^{w}\left(X\right) 
		\end{equation}
		and hence it follows that $\ell_{\infty}^{mid}(X)=\ell_{\infty}(X)$ for all Banach spaces $X$.
		
	Recently, Botelho and Campos\cite{Botelho1} have introduced a new notion of dual for a vector-valued sequence class $S$ as,
	\begin{equation}\label{eq:4}
		S^{dual}(X)=\left\{\left(x_{j}\right)_{j}\in X^{\mathbb{N}}:\sum_{j=1}^{\infty}x^{*}_{j}\left(x_{j}\right) \text{converges for every} (x^{*}_{j})_{j}\in S(X^{*}) \right\}.
	\end{equation}
	In particular, the dual of the class of mid $p$-summable sequences can be defined as,
	$$\left(\ell_{p}^{mid}\right)^{dual}(X)=\left\{(x_{j})_{j}\in X^{\mathbb{N}}:\sum_{j=1}^{\infty}|x^{*}_{j}(x_{j})|<\infty, \text{ for each } (x^{*}_{j})_{j} \in \ell_{p}^{mid}(X^{*})\right\}.$$
	and, $\left(\ell_{p}^{mid}\right)^{dual}(X)$  is a Banach space endowed with the norm,
	\begin{equation*}
		\|(x_{j})_{j}\|_{p,mid}^{dual}=\sup_{(x^{*}_{j})_{j}\in B_{\ell_{p}^{mid}(X^{*})}}\sum_{j=1}^{\infty}|x^{*}_{j}(x_{j})|.
	\end{equation*}
	Let us note that the following inclusion holds for all Banach spaces $X$: 
	\begin{equation}\label{eq:5}
		\ell_{p^{*}}\left<X\right>\subseteq \left(\ell_{p}^{mid}\right)^{dual}(X) \subseteq \ell_{p^{*}}(X).
	\end{equation} 
	For improving the summability of sequences, several operator ideals have been defined and studied extensively; for example, the ideal of absolutely $p$-summing operators, Cohen $p$-summing operators, weakly mid $p$-summing operators etc.  
		\begin{Definition}(\cite{DiJaTo})
			A linear operator $T:X\rightarrow Y$ is said to  be \textit{absolutely $p$-summing} if the operator $\hat{T}:\ell_{p}^{w}(X)\rightarrow \ell_{p}(Y)$, defined by
			\begin{equation*}
				\hat{T}((x_{j})_{j}) = (T(x_{j}))_{j}, \text{ where }(x_{j})_{j}\in\ell_{p}^{w}(X)
			\end{equation*} is well defined and continuous.
		\end{Definition}
		The set of all absolutely $p$-summing operators from $X$ to $Y$ denoted by $\Pi_{p}\left(X,Y\right)$ is a Banach space endowed with the norm $\pi_{p}(T)=\|\hat{T}\|$ for all Banach spaces $X$ and $Y$. Futhermore, $[\Pi_{p},\pi_{p}]$ is a maximal Banach operator ideal (see \cite[17.1.3]{DiJaTo}).

		\begin{Definition}(\cite{Botelho})
			A linear operator $T:X\rightarrow Y$ is said to be \textit{weakly mid $p$-summing} if the map $\hat{T}:\ell_{p}^{w}(X)\rightarrow \ell_{p}^{mid}(Y)$, defined by
			\begin{equation*}
				\hat{T}((x_{j})_{j}) = (T(x_{j}))_{j}, \text{ where } (x_{j})_{j}\in\ell_{p}^{w}(X)
			\end{equation*} is well defined and continuous.
		\end{Definition}
		The class of weakly  mid $p$-summing operators from $X$ to $Y$ is denoted by $W_{p}^{mid}\left(X,Y\right)$. 	The next result proved in \cite{Karn} shows a close relationship between the class of weakly mid $p$-summing operators and absolutely $p$-summing operators.
		\begin{thm}\label{thm:Wp and Pip}
			For $1\leq p <\infty$ and $T\in \mathcal{L}(X,Y)$, the following statements are equivalent:
			
			(a) $T$ is weakly mid $p$-summing.
			
			(b) $ST\in \Pi_p(X,\ell_p)$ for each $S\in \mathcal{L}(Y, \ell_p)$.
			
			(c) $TU \in \Pi_p^d(\ell_p^*, Y)$ for each $U\in \mathcal{L}(\ell_p^*, X)$.
		\end{thm}
		
		It has been proved in \cite{Karn} that $[W_p^{mid},w_{p}^{mid}]$ is a normed operator ideal endowed with the norm
	\begin{equation*}
		w_{p}^{mid}(T)= \sup_{S\in B_{\mathcal{L}(Y,\ell_{p})}}\pi_{p}(ST).
	\end{equation*}
The completeness of the normed operator ideal $[W_{p}^{mid},w_{p}^{mid}]$ has been established in \cite{Fourie}. Using $w_{p}^{mid}(T)=\|\hat{T}\|$, the completeness of $[W_{p}^{mid},w_{p}^{mid}]$ has also been obtained in \cite{Botelho}. Also note that $[W_{p}^{mid},w_{p}^{mid}]$ is a maximal operator ideal (see \cite{zee}).
\begin{Definition}(\cite{Botelho})
A linear operator $T:X\rightarrow Y$ is said to be absolutely mid $p$-summing if $\hat{T}:\ell_{p}^{mid}(X)\rightarrow \ell_{p}(Y)$ defined as
	\begin{equation}\label{eq:14}
		\hat{T}((x_{j})_{j})= (Tx_{j})_{j}, \text{ where } (x_{j})_{j}\in \ell_{p}^{mid}(X)
	\end{equation}
	is well defined and continuous.
\end{Definition}
The class of all absolutely mid $p$-summing operators from $X$ to $Y$ denoted by $\Pi_{p}^{mid}(X,Y)$ is a Banach space endowed with the norm $\pi_{p}^{mid}(T)=\|\hat{T}\|$. Thus $[\Pi_{p}^{mid},\pi_{p}^{mid}]$ is a Banach operator ideal .
\\\\Generalizing an earlier work of Saphar\cite{Sap1}, Chevet \cite{che} and Saphar\cite{Sap2} have independently introduced the norms $d_{p}$ and $g_{p}$ as follows:	
\begin{align}\label{eq:1.7}
\begin{split}
	d_{p}\left(u\right)&= \inf\left\{\|(x_{j})_{j=1}^{n}\|_{p^{*}}^{w}\|(y_{j})_{j=1}^{n}\|_{p}:u=\sum_{j=1}^{n}x_{j}\otimes y_{j}\right\}\\
	\textnormal{and}\\
	g_{p}\left(u\right)&= \inf\left\{\|(x_{j})_{j=1}^{n}\|_{p}\|(y_{j})_{j=1}^{n}\|_{p^{*}}^{w}:u=\sum_{j=1}^{n}x_{j}\otimes y_{j}\right\}.
\end{split}
\end{align}

These norms are known as \textit{Chevet-Saphar tensor norms}. Note that $d_{p}$ and $g_{p}$ are transposes of each other and satisfy the following duality relations with the ideal of absolutely $p$-summing operators:
\begin{align}\label{eq:32}
\begin{split}
\left(X\hat{\otimes}_{d_{p}}Y\right)^{*}&\cong\Pi_{p^{*}}\left(X,Y^{*}\right)\\
\textnormal{and}\\
\left(X\hat{\otimes}_{g_{p}}Y\right)^{*}&\cong\Pi_{p^{*}}\left(Y,X^{*}\right) 
\end{split}
\end{align}
 where $X\hat{\otimes}_{d_{p}}Y$ and $X\hat{\otimes}_{g_{p}}Y$ denote the completion of the tensor products $X\otimes_{d_{p}}Y$ and $X\otimes_{g_{p}}Y$ respectively.

In this paper, we define a tensor norm $\alpha_{p}$ corresponding  to the operator ideal $W_{p}^{mid}$ and prove that $W_{p}^{mid}(X,Y^{*})\cong (X\hat{\otimes}_{\alpha_{p}}Y)^{*}$ for all Banach spaces $X$ and $Y$, and $1\leq p\leq\infty$.  The basic idea to define these tensor norms is an extension of the following interpretation of Chevet-Saphar tensor norms: 
\begin{align}\label{eq:1.9}
\begin{split}
d_{p}\left(u\right)&= \inf\left\{\|(x_{j})_{j=1}^{n}\|_{p^{*}}^{w}\|(y_{j})_{j=1}^{n}\|_{p^{*}}^{dual}:u=\sum_{j=1}^{n}x_{j}\otimes y_{j}\right\}\\
\textnormal{and}\\
g_{p}\left(u\right)&= \inf\left\{\|(x_{j})_{j=1}^{n}\|_{p^{*}}^{dual}\|(y_{j})_{j=1}^{n}\|_{p^{*}}^{w}:u=\sum_{j=1}^{n}x_{j}\otimes y_{j}\right\}.
\end{split}
\end{align}
It is noteworthy that \eqref{eq:1.9} is obtained from \eqref{eq:1.7}  by using the isometric isomorphism
$\left(\ell_{p^{*}}(X),\|\cdot\|_{p^{*}}\right) \cong\left(\ell_{p}^{dual}(X),\|\cdot\|_{p}^{dual}\right).$
The last section is devoted to the study of the operator ideal $\Pi_{p}^{mid}$; indeed, a tensor norm representation for $\Pi_{p}^{mid}$ has been obtained. Finally using the tensor norm representation of $\Pi_{p}^{mid}$, we characterize the adjoints of the absolutely mid $p$-summing operators as well as the operators whose adjoints are absolutely mid $p$-summing. 
 \section{Ideal of weakly mid $p$-summing operators}
In this section, we will define a tensor norm  corresponding to the operator ideal of weakly mid $p$-summing operators and obtain the dual ideal of $W_{p}^{mid}$ using this tensor norm. 

Let  $X,Y$ be Banach spaces and $1\leq p\leq \infty$. For every $u\in X\otimes Y$, we define,

\begin{equation}\label{eq:6}
	\alpha_{p}\left(u\right) = \inf\left\{\|(x_{j})_{j=1}^{n}\|_{p}^{w}\|(y_{j})_{j=1}^{n}\|_{p,mid}^{dual}:u=\sum_{j=1}^{n}x_{j}\otimes y_{j}\right\}
\end{equation}
where
$\|\cdot\|_{p,mid}^{dual}$ is the norm of the dual space $(\ell_{p}^{mid})^{dual}\left(\cdot\right)$. Let us note that 
\begin{equation}\label{eq:2.2}
	 d_{p^{*}}\left(\cdot\right)\leq \alpha_{p}\left(\cdot\right)
\end{equation} since  $\|\cdot\|_{p^{*}}\leq \|\cdot\|_{p,mid}^{dual}$ for each $1\leq p\leq \infty$.  
The first step towards proving $\alpha_{p}$ is a tensor norm is to prove that it is a reasonable cross norm.
\begin{pro}\label{thm:crossnorm alphap}
	For $1\leq p \leq\infty, \alpha_{p}$ is a reasonable cross norm on $X\otimes Y.$
\end{pro}
\begin{proof}
	
	It is evident that for any $\lambda \in \mathbb{K}$ and $u\in X\otimes Y$, $\alpha_{p}\left(\lambda u\right)=|\lambda|\alpha_{p}\left(u\right)$.

	Let $u_{1},u_{2}\in X\otimes Y$ and $\epsilon>0$. Since $\alpha_{p}$ is the infimum, we can find a representation of $u_{i}=\sum\limits_{j=1}^{n}x_{ij}\otimes y_{ij}$ such that $\|(x_{ij})_{j=1}^{n}\|_{p}^{w}\leq \left(\alpha_{p}\left(u_{i}\right)+\epsilon\right)^{1/p}$
	and $\|(y_{ij})_{j=1}^{n}\|_{p,mid}^{dual}\leq \left(\alpha_{p}\left(u_{i}\right)+\epsilon\right)^{1/p^{*}}$, where $i=1,2$. Concatenating the sequences $(x_{ij})_{j},(y_{ij})_{j}$ for $i=1,2$, we can write
	\begin{equation*}
	u_1+u_2=\sum_{i,j}x_{ij}\otimes y_{ij}
	\end{equation*}
where
	\begin{align*}
	\|(x_{ij})_{i,j}\|_{p}^{w}\|(y_{ij})_{i,j}\|_{p,mid}^{dual}&\leq \left(\alpha_{p}\left(u_{1}\right)+\alpha_{p}\left(u_{2}\right)+2\epsilon\right)^{1/p}\left(\alpha_{p}\left(u_{1}\right)+\alpha_{p}\left(u_{2}\right)+2\epsilon\right)^{1/p^{*}}\\
	&\leq \left(\alpha_{p}\left(u_{1}\right)+\alpha_{p}\left(u_{2}\right)+2\epsilon\right).
	\end{align*}
	Letting $\epsilon$ tend to zero, we obtain $\alpha_{p}(u_1+u_2)\leq \alpha_{p}(u_1)+\alpha_{p}(u_2)$.
	
	Note that  $\alpha_{p}(x\otimes y)=\|x\|\|y\|$ for all $x\in X$ and $y\in Y$. Therefore $\alpha_{p}\left(\cdot\right)\leq \pi\left(\cdot\right)$ by using the triangle inequality for $\alpha_p$. Also, it is clear from \eqref{eq:2.2}  that $\varepsilon\left(\cdot\right)\leq  \alpha_{p}\left(\cdot\right)$. Thus $\alpha_{p}(u)=0$ if and only if $u=0$. 
	This also shows that $\alpha_{p}$ is a reasonable cross norm. 
\end{proof} 
 From the definition of $\alpha_{p}$, it is clear that $\alpha_{p}$ is finitely generated. Also, the linear stability of the sequence spaces $\ell_{p}^{w}\left(\cdot\right)$ and $(\ell_{p}^{mid})^{dual}\left(\cdot\right)$ gives the uniform property of the finitely generated reasonable cross norm $\alpha_{p}$. Consequently,
\begin{thm}\label{tensor norm alphap}
	For $1\leq p\leq\infty, ~\alpha_{p}$ is a finitely generated tensor norm.
\end{thm}
Next we prove that the space $(X\otimes_{\alpha_{p}}Y)^{*}$ can be identified with the class of weakly mid $p$-summing operators from $X$ to $Y^{*}$. We begin with the following lemma. 

\begin{Lemma}\label{lpmid dual for Lp spaces}
Let $1\leq p<\infty$. Then $\ell_{p^{*}}\left(X\right)=(\ell_{p}^{mid})^{dual}\left(X\right)$ if  $X^{*}$ is a subspace of $L_{p}(\mu)$ for some Borel measure $\mu$. Furthermore, $\|(x_{j})_{j}\|_{p^{*}}= \|(x_{j})_{j}\|_{p,mid}^{dual}$ for every $(x_{j})_{j}\in X$.
\end{Lemma}

\begin{proof}
Consider $(x_{j})_{j}\in \ell_{p^{*}}(X)$. Using \cite[Theorem 4.5]{Karn},  $\ell_{p}(X^{*})= \ell_{p}^{mid}(X^{*})$, and therefore
\begin{equation*}
\sum_{j=1}^{\infty}\left|\phi_{j}(x_{j})\right| \leq \left(\sum_{j=1}^{\infty}\|\phi_{j}\|^{p}\right)^{\frac{1}{p}}\left(\sum_{j=1}^{\infty}\|x_{j}\|^{p^{*}}\right)^{\frac{1}{p^{*}}}
< \infty,  \text{ for all } (\phi_{j})_{j}\in \ell_{p}^{mid}(X^{*}).
\end{equation*} 
Thus $\ell_{p^{*}}\left(X\right)\subset(\ell_{p}^{mid})^{dual}\left(X\right)$. The reverse inclusion follows from $\eqref{eq:5}$. 	  	 
\end{proof}
\begin{thm}\label{Wp and alphap}
	Let $1\leq p\leq\infty$ and $X$, $Y$ be Banach spaces. An operator $T:X \rightarrow Y^{*}$ corresponds to a bounded linear functional on $X\hat{\otimes}_{\alpha_{p}}Y$ if and only if $T\in W_{p}^{mid}(X,Y^{*})$. The operator norm of the bounded linear functional corresponding to $T$  in $\left(X\hat{\otimes}_{\alpha_{p}}Y\right)^{*}$ is equal to $w_{p}^{mid}(T)$. 	
\end{thm}
\begin{proof}
 We will assume $1\leq p< \infty$; the case $p=\infty$ follows from \eqref{eq:2.2} and \cite[Proposition 6.6 ]{ryan} as $\alpha_{\infty}=\pi$ and $W_{\infty}^{mid}(X,Y^*)=\mathcal{L}(X,Y^*)$. Let $\Phi \in \left(X\hat{\otimes}_{\alpha_{p}}Y\right)^{*}$. Then  define $T\equiv T_{\Phi}:X\rightarrow Y^{*}$ as $\left<Tx,y\right>=\Phi(x\otimes y)$ for all $x\in X,~ y\in Y$.

	In order to prove that $T \in W_{p}^{mid}(X,Y^{*})$, we need to show that $ST\in\Pi_p(X, \ell_{p})$ for each $S\in \mathcal{L}(Y^{*},\ell_{p})$ in view of Theorem \ref{thm:Wp and Pip}, or equivalently \eqref{eq:32} imply that it is enough to show that $ST$ can be identified with a continuous linear functional defined on $ X\hat{\otimes}_{d_{p^{*}}}\ell_{p^{*}}$.

Define $^{*}\Phi :X^{**}\hat{\otimes}_{\alpha_{p}}Y^{**}\rightarrow \mathbb{K}$ as 
$^*\Phi(x^{**}\otimes y^{**}) =<T^*y^{**}, x^{**}>,~x^{**}\in X^{**},~ y^{**}\in Y^{**}$. Note that $^{*}\Phi$ is the continuous linear functional corresponding to the canonical left extension of the bilinear form $B_{\Phi}$, where $B_{\Phi}:X\times Y\rightarrow \mathbb{K}$ is defined as 
$B_{\Phi}(x,y)=\Phi(x \otimes y)$ for all $x\in X$ and $y\in Y$.
 Then for any $S\in \mathcal{L}(Y^{*},\ell_{p})$ and $u=\sum\limits_{j=1}^{n}x_{j}\otimes\beta_{j}\in X\otimes \ell_{p^{*}}$,  we have
	\begin{align*}
		\left|\sum_{j=1}^{n}\left<STx_{j},\beta_{j}\right>\right|&=\left|\sum_{j=1}^{n}\left<Tx_{j},S^{*}\beta_{j}\right>\right|= \left|^{*}\Phi\left(\sum_{j=1}^{n}x_{j}\otimes S^{*}\beta_{j}\right)\right|\\
		&\leq \left\|\Phi\right\|\left\|S^{*}\right\|\left
		\|(x_{j})_{j}\right\|_{p}^{w}\left\|(\beta_{j})_{j}\right\|_{p,mid}^{dual}\\
		&= \left\|\Phi\right\|\left\|S^{*}\right\|\left
		\|(x_{j})_{j}\right\|_{p}^{w}\left\|(\beta_{j})_{j}\right\|_{p^{*}}
	\end{align*}
by \cite[Theorem 6.5]{ryan} and Lemma \ref{lpmid dual for Lp spaces}.
 Taking infimum over all representations of $u$, we obtain the inequality,
\begin{equation}
	\left|\left<ST,u\right>\right|\leq \|\Phi\|\|S^{*}\|d_{p^{*}}(u).
	\end{equation}
	Thus $ST\in\Pi_p(X, \ell_{p})$ and $\pi_p(ST)\leq \|\Phi\|\|S\|$.
Also, note that	\begin{align} \label{eq:7}
		w_{p}^{mid}\left(T\right)=\sup_{S\in B_{\mathcal{L}(Y^{*},\ell_{p})}}\pi_{p}\left(ST\right)
= \sup_{S\in B_{\mathcal{L}(Y^{*},\ell_{p})}} \sup_{u\in B_{X\otimes_{d_{p^{*}}} \ell_{p^{*}}}}\left|\left<ST,u\right>\right|\leq\|\Phi\|.
	\end{align}
On the other hand, for $T\in W_{p}^{mid}(X,Y^{*})$, define $\Phi\equiv\Phi_{T}:X\otimes_{\alpha_{p}}Y\rightarrow \mathbb{K}$ as $\Phi(u)= \sum\limits_{j=1}^{n}\left<Tx_{j}, y_{j}\right>$ for each $u=\sum\limits_{j=1}^{n}x_{j}\otimes y_{j} \in X\otimes_{\alpha_{p}} Y$. Then by using the definition of $\|.\|_{p,mid}^{dual}$-norm, we have
	\begin{align*}
		\left|\Phi\left(u\right)\right|&\leq \sum_{j=1}^{n}\left|\left<Tx_{j},y_{j}\right>\right|
		= \|(x_{j})_{j=1}^{n}\|_{p}^{w}\|(y_{j})_{j=1}^{n}\|_{p,mid}^{dual}\sum_{j=1}^{n}\left|\left<T\left(\frac{x_{j}}{\|(x_{j})_{j=1}^{n}\|_{p}^{w}}\right),\frac{y_{j}}{\|(y_{j})_{j=1}^{n}\|_{p,mid}^{dual}}\right>\right|\\
		& \leq \|(x_{j})_{j}\|_{p}^{w}\|(y_{j})_{j}\|_{p,mid}^{dual} \left\|\left(T\left(\frac{x_{j}}{\|(x_{j})_{j=1}^{n}\|_{p}^{w}}\right)\right)_{j=1}^{n}\right\|_{p}^{mid}\left\|\left(\frac{y_{j}}{\|(y_{j})_{j=1}^{n}\|_{p,mid}^{dual}}\right)_{j=1}^{n}\right\|_{p,mid}^{dual}.\\
	\end{align*}
	Taking infimum over all representations of $u$, we get
	\begin{equation}\label{eq:8}
		\left|\Phi\left(u\right)\right|\leq w_{p}^{mid}\left(T\right)\alpha_{p}\left(u\right).
	\end{equation}
 Thus $\Phi\in \left(X\hat{\otimes}_{\alpha_{p}}Y\right)^{*}$ and
	$
		w_{p}^{mid}\left(T\right)=\|\Phi\|
	$
by	using $\eqref{eq:7}$ and $\eqref{eq:8}$.
\end{proof}
The next proposition gives some equivalent descriptions for weakly mid $p$-summing operators.
\begin{pro}
	Let $1\leq p\leq\infty$. Then the following are equivalent:
	\begin{enumerate}[(a)]
		\item $T\in W_{p}^{mid}(X,Y)$.
		\item There exists $C>0$ such that,
		\begin{equation}\label{eq:9}
			\left|\sum_{j=1}^{n}\left<Tx_{j}\,f_{j}\right>\right|\leq C \|(x_{j})_{j=1}^{n}\|_{p}^{w}\|(f_{j})_{j=1}^{n}\|_{p,mid}^{dual}
		\end{equation}
		for every finite sequences $\left(x_{1},x_{2},...,x_{n}\right)$ and $\left(f_{1},f_{2},...,f_{n}\right)$ in $X$ and $Y^{*}$ respectively.
		\item There exists $C>0$ such that,
		\begin{equation}\label{eq:10}
			\|(Tx_{j})_{j=1}^{n}\|_{p}^{mid}\leq C \|(x_{j})_{j=1}^{n}\|_{p}^{w}
		\end{equation}
		for every finite sequence $\left(x_{1},x_{2}\,...,x_{n}\right)$
		in $X$.
		\item The map $\hat{T}:\ell_{p}^{w}(X)\rightarrow \ell_{p}^{mid}(Y)$ is well defined and continuous.
		\item There exists $C>0$ such that,
		\begin{equation}\label{eq:11}
			\left(\sum_{m=1}^{\infty}\sum_{j=1}^{n}\left|f_{m}(Tx_{j})\right|^{p}\right)^{1/p}\leq C \|(x_{j})_{j=1}^{n}\|_{p}^{w} \|(f_{m})_{m=1}^{\infty}\|_{p}^{w}
		\end{equation}
		for every finite sequence  $\left(x_{1},x_{2}\,...,x_{n}\right)$
		in $X$ and $\left(f_{m}\right)_{m}$
		in $\ell_{p}^{w}\left(Y^{*}\right)$.
		\item There exists $C>0$ such that,
		\begin{equation}\label{eq:12}
			\left(\sum_{m=1}^{\infty}\sum_{j=1}^{\infty}\left|f_{m}(Tx_{j})\right|^{p}\right)^{1/p}\leq C \|(x_{j})_{j}\|_{p}^{w} \|(f_{m})_{m}\|_{p}^{w}
		\end{equation}
		for every sequences $\left(x_{j}\right)_{j}$ in $\ell_{p}^{w}(X)$ and $\left(f_{m}\right)_{m}$
		in $\ell_{p}^{w}\left(Y^{*}\right)$.
	\end{enumerate}
	Furthermore,\\
$
			w_{p}^{mid}(T)=\|\hat{T}\|=\inf\{C:\eqref{eq:9}  \textnormal{ holds} \}= \inf\{C:\eqref{eq:10} \textnormal{ holds}\}= \inf\{C:\eqref{eq:11} \textnormal{ holds}\} =\inf\{C:\eqref{eq:12} \textnormal{ holds}\}.
	$

\end{pro}
\begin{proof}
 $(a)\iff (b)\iff (c)\iff(d)$ is clear from Theorem \ref{Wp and alphap}.
 
 $(d) \iff (e) \iff (f)$ follows from the definition of $\|\cdot\|_p^{mid}$-norm.

\end{proof}
Let us now recall the following notion of the dual of an operator  ideal from \cite{DiJaTo}. 

\begin{Definition}
	Let $[\mathcal{U}, \|\cdot\|_{\mathcal{U}}]$ be a Banach operator ideal. Then the components of the dual ideal of $\mathcal{U}$ is  given by,
	\begin{equation}\label{eq:13}
		\mathcal{U}^{dual}(X,Y)=\{T:X\rightarrow Y: T^{*}\in \mathcal{U}(Y^{*},X^{*})\}.
	\end{equation}
	Moreover, $[\mathcal{U}^{dual}, \|\cdot\|_{\mathcal{U}}^{dual}]$ is a Banach operator ideal, where $\|T\|_{\mathcal{U}}^{dual}= \|T^{*}\|_{\mathcal{U}}$.
\end{Definition}

In \cite{Botelho1}, using the spherical completeness property of sequence classes, the authors have developed a unified approach to characterize the adjoints of operators defined by the transformation of certain kind of summable sequences. 
In particular, one can obtain a description of the dual ideal of $W_{p}^{mid}$  using the spherical completeness property of $\ell_{p}^{w}(\cdot)$ and $\ell_{p}^{mid}(\cdot)$. Here, we  give an alternate proof for the same using the tensor norm $\alpha_{p}$.

\begin{pro}
	Let  $1\leq p <\infty$ and $T:X\rightarrow Y$ be a continuous linear operator. Then $\hat{T}:(\ell_{p}^{mid})^{dual}\left(X\right) \rightarrow \ell_{p^{*}}\left<Y\right>$ is well defined and continuous if and only if its adjoint   $T^{*}\in W_{p}^{mid}(Y^*, X^*)$. Furthermore, $w_p^{mid}(T^*)=\|\hat{T}\|$.
\end{pro}
\begin{proof}
	Let $T\in \mathcal{L}(X,Y)$ be such that $\hat{T}:(\ell_{p}^{mid})^{dual}(X)\rightarrow \ell_{p^{*}}\left<Y\right>$ is well defined and continuous. Define $\Phi_{T^{*}}: Y^{*}\otimes X\rightarrow \mathbb{K}$ as  $\Phi_{T^{*}}(f\otimes x)= \left<T^{*}f,x\right>$ for all $f\in Y^*$ and $x\in X.$ Then $\Phi_{T^*}$ is linear and satisfies the inequality
	\begin{align*}
		\left|\Phi_{T^{*}}\left(u\right)\right|&= \left|\sum_{j=1}^{n}\left<T^{*}f_{j},x_{j}\right>\right|=\left|\sum_{j=1}^{n}\left<f_{j},Tx_{j}\right>\right|\\
		&\leq \|\hat{T}\|\|\left(f_{j}\right)_{j}\|_{p}^{w} \|\left(x_{j}\right)_{j}\|_{p,mid}^{dual}
				\end{align*}
			for all representations of $u=\sum\limits_{j=1}^{n}f_{j}\otimes x_{j}\in Y^{*}\otimes X$. Thus $\left|\Phi_{T^{*}}\left(u\right)\right| \leq \|\hat{T}\|\alpha_p(u)$ for each $u\in Y^{*}\otimes X$
and hence $\Phi_{T^*}\in (Y^{*}\otimes X)^*$ with $\|\Phi_{T^*}\|\leq \|\hat{T}\|$ or equivalently, $T^{*}\in W_{p}^{mid}(Y^*, X^*)$ with
\begin{equation}\label{eq:2.11}
w_p^{mid}(T^*)\leq \|\hat{T}\|
\end{equation}
	by using Theorem \ref{Wp and alphap}.
	
	 Conversely, let $T^{*}\in W_{p}^{mid}(Y^*, X^*)$. Then by Theorem \ref{Wp and alphap},  $\Phi:Y^{*}\hat{\otimes}_{\alpha_{p}}X\rightarrow \mathbb{K}$ defined as $\Phi(f\otimes x)= \left<T^{*}f,x\right>$ is a continuous linear functional with $\|\Phi\|=w_p^{mid}(T^*)$. Therefore using the spherical completeness property of $\ell_{p}^{w}(\cdot)$ (see \cite[Lemma 2.2]{Botelho1}), we get
	\begin{align*}
		\sum_{j=1}^{n}\left|f_{j}(Tx_{j})\right|&= \left|\sum_{j=1}^{n}f_{j}(Tx_{j})\right| =\left|\Phi\left(\sum_{j=1}^{n}f_{j}\otimes x_{j}\right)\right|\\
		&\leq w_p^{mid}(T^*) \|\left(f_{j}\right)_{j}\|_{p}^{w} \|\left(x_{j}\right)_{j}\|_{p,mid}^{dual}.
	\end{align*}
This proves that $\hat{T}:(\ell_{p}^{mid})^{dual}(X)\rightarrow \ell_{p^{*}}\left<Y\right>$ is continuous with $\|\hat{T}\|\leq w_p^{mid}(T^*)$ and hence $\|\hat{T}\|=w_p^{mid}(T^*)$ by using \eqref{eq:2.11}. 
 
\end{proof}
\begin{cor} Let $X$, $Y$ be Banach spaces and $1\leq p<\infty$. Then the component $(W_{p}^{mid})^{dual}(X,Y)$  of the dual ideal  $(W_{p}^{mid})^{dual}$ is  given by,
	  $$(W_{p}^{mid})^{dual}(X,Y)=\{T\in \mathcal{L}(X,Y): \hat{T}: (\ell_{p}^{mid})^{dual}(X)\rightarrow \ell_{p^*}\left<Y\right> \text{is well defined and continuous}\}.$$ 
Moreover, $(w_{p}^{mid})^{dual}(T)=\|\hat{T}\|$ for all $T\in (W_{p}^{mid})^{dual}(X,Y)$.
\end{cor}

	\section{Absolutely mid $p$-summing operators}
In this section we study the class of absolutely mid $p$-summing operators and obtain a tensor norm associated to this operator ideal.

For $1\leq p\leq\infty$ and Banach spaces $X,Y$, we define \begin{equation}\label{eq:15}
	\gamma_{p}(u) = \inf\left\{\|(x_{j})_{j=1}^{n}\|_{p^{*}}^{mid}\|(y_{j})_{j=1}^{n}\|_{p}:u=\sum_{j=1}^{n}x_{j}\otimes y_{j}\right\}.
\end{equation} 
\begin{pro}\label{dp and gammap}
	Let $X$ and $Y$ be Banach spaces, we have 
	\begin{enumerate}[(i)]
		\item $d_{p}\leq \gamma_{p}$ for each $1\leq p\leq\infty$.
		\item $\gamma_{1}=d_{1}=g_{1}=\pi$.
	\end{enumerate}
\end{pro}
\begin{proof}
(i) follows by using the inequality	$\|\cdot\|_{p}^{w}\leq \|\cdot\|_{p}^{mid}$ and, (ii) can be proved easily using the identity $\ell_{\infty}^{mid}(\cdot)=\ell_{\infty}(\cdot)$ and \cite[Proposition 6.6]{ryan}.
\end{proof}
\begin{thm}\label{tensor norm gammap}
For $1\leq p\leq\infty, \gamma_{p}$ is a finitely generated tensor norm on $X\otimes Y$.
\end{thm}
\begin{proof}
	The proof is analogous to Theorem \ref{tensor norm alphap}.
\end{proof}
The transpose of $\gamma_{p}$ is defined as follows:
\begin{equation}\label{eq:16}
	\delta_{p}(u)= \inf\{\|(x_{j})_{j=1}^{n}\|_{p}\|(y_{j})_{j=1}^{n}\|_{p^{*}}^{mid}:u=\sum_{j=1}^{n}x_{j}\otimes y_{j}\}.
\end{equation}
It can be easily checked that  $\delta_{p}$ is a tensor norm for $p\in [1,\infty]$.

Next, we prove 
\begin{pro}\label{Projective norm}
	The tensor norm $\gamma_{p}$ is right projective  and $\delta_{p}$ is left projective for every $p\in [1,\infty)$.
\end{pro}
\begin{proof}
	We will only prove that $\gamma_{p}$ is right projective, the proof of $\delta_{p}$ is left projective follows analogously. Let $X,Y,Z$ be Banach spaces and $Q:Z\rightarrow Y$ be a quotient operator .	We need to prove that $I_X\otimes Q: X\otimes_{\gamma_{p}}Z\rightarrow X\otimes_{\gamma_{p}}Y$ is a quotient operator, where  $I_X$ is the identity operator on $X$. Consider $u=\sum\limits_{j=1}^{n}x_{j}\otimes z_{j} \in X\otimes_{\gamma_{p}}Z$ with $\gamma_{p}(u)\leq 1$. Then, 
	\begin{align*}
	\gamma_{p}(I_X\otimes Q(u)) &= \gamma_{p}(\sum_{j=1}^{n}x_{j}\otimes Qz_{j}) \\
	&\leq \|(x_{j})_{j=1}^{n}\|_{p^{*}}^{mid}\|(Qz_{j})_{j=1}^{n}\|_{p}\\
	&\leq \|(x_{j})_{j=1}^{n}\|_{p^{*}}^{mid}\|(z_{j})_{j=1}^{n}\|_{p} \textnormal{ ~~since $\|Q\|=1$}\\
	&\leq	\gamma_{p}(u)
	\end{align*} 
	after taking infimum over all representations of $u$.
	
	Let $v\in B_{X\otimes_{\gamma_{p}}Y}$. Then $\gamma_p(v)\leq 1$ and we can choose a representation of $v=\sum\limits_{j=1}^{n}x_{j}\otimes y_{j}$ with  $\|(x_{j})_{j=1}^{n}\|_{p^{*}}^{mid}\|(y_{j})_{j=1}^{n}\|_{p}<1$. Since $Q$ is a quotient operator, there exist $\epsilon>0$ and $z_{j}\in Z$ such that $Qz_{j}=y_{j}$ and $\|z_{j}\|\leq (1+\epsilon)\|y_{j}\|$ for each $1\leq j\leq n$. Now, define $u= \sum\limits_{j=1}^{n}x_{j}\otimes z_{j} \in X\otimes Z$. Note that $I\otimes Q(u)=v$ and
	$
	\gamma_{p}(u)\leq (1+\epsilon) \|(x_{j})_{j=1}^{n}\|_{p^{*}}^{mid}\|(y_{j})_{j=1}^{n}\|_{p}.
$
	Since $\epsilon$ is arbitrarily chosen, $	\gamma_{p}(u)\leq 1$. Consequently, $\|I\otimes Q\|=1$ and hence $\gamma_{p}$ is a right projective tensor norm.

\end{proof}
Recall from \cite{ryan} that the Chevet-Saphar tensor norms satisfy the inequality $d_{p}(u) \leq \|u\|_{p}\leq g_{p}(u)$  for each $u\in \ell_{p}\otimes X$ for any Banach space $X$.  Further, including the tensor norms $\gamma_{p}$ and $\delta_{p}$, we obtain the following result:

\begin{pro}\label{norm p and gammap}
 For any $1\leq p<\infty$ and Banach space $X$, the following inequality holds for each $u\in \ell_{p}\otimes X$:
 \begin{equation}\label{eq:17}
 	\gamma_{p}(u)\leq \|u\|_{p}\leq \delta_{p}(u).
 \end{equation}
\end{pro}
\begin{proof}
	Let $1\leq p<\infty$ and $u=\sum\limits_{j=1}^{n}a_{j}\otimes x_{j}\in\ell_{p}\otimes X$, where $a_{j}=(a_{jk})_{k}$. Then we can write $u=\sum\limits_{k=1}^{\infty}e_{k}\otimes u_{k}$ with  $u_{k}=\sum\limits_{j=1}^{n}a_{jk}x_{j}$. Note that $u$ can be identified with the element $(u_{k})_{k}$ in $\ell_{p}(X)$. Since $\|(e_{k})_{k}\|_{p^{*}}^{mid}=1$, we have 
	\begin{equation}
		\gamma_{p}(u)\leq\|u\|_{p}.
	\end{equation}
	Proceeding as in \cite[Example 6.8]{ryan}, 

	\begin{equation}\label{eq:34}
	\|u\|_{p}\leq \|(a_{j})_{j}\|_{p}\|(x_{j})_{j}\|_{p^{*}}^{w}\leq \|(a_{j})_{j}\|_{p}\|(x_{j})_{j}\|_{p^{*}}^{mid}.
	\end{equation}
	Since \eqref{eq:34} holds for every representation of $u$, we get
	\begin{equation}
		\gamma_{p}(u)\leq \|u\|_{p}\leq \delta_{p}(u)
	\end{equation}
	for every $u\in \ell_{p}\otimes X$.
\end{proof}
\begin{Lemma}\label{Convergence}
	Let $X$ and $Y$ be Banach spaces. For $1\leq p\leq \infty$, $(x_{j})_{j}\in \ell_{p^{*}}^{mid}(X)$ and $(y_{j})_{j}\in \ell_{p}(Y)$, then the series $\sum\limits_{j=1}^{\infty}x_{j}\otimes y_{j}$ converges in $\hat{X\otimes_{\gamma_{p}}} Y.$
\end{Lemma}
\begin{proof}
	Let us note that for finite sequences $(x_{k},x_{k+1},...,x_{n})\subset X$ and $(y_{k},y_{k+1},...,y_{n})\subset Y$, 
	\begin{equation}
		\gamma_{p}\left(\sum_{j=k}^{n}x_{j}\otimes y_{j}\right)\leq \|(x_{j})_{j}\|_{p^{*}}^{mid}\left(\sum_{j=k}^{n}\|y_{j}\|^{p}\right)^{1/p}	
	\end{equation}
	holds for any $k,n\in \mathbb{N}$ with $k\leq n$.
	Thus the series  $\sum\limits_{j=1}^{\infty}x_{j}\otimes y_{j}$ satisfies Cauchy's criterion and hence convergent in  $\hat{X\otimes_{\gamma_{p}}} Y$. 
\end{proof}

\begin{pro}\label{Existence of sequences}
Let $X$ and $Y$ be Banach spaces and $1\leq p\leq\infty.$ Then for each $u\in X\hat{\otimes}_{\gamma_{p}}Y$, there exist sequences $(x_{n}) _{n}\in \ell_{p^{*}}^{mid}(X)$ and $(y_{n})_{n}\in \ell_{p}(Y)$ such that  $\sum\limits_{n=1}^{\infty}x_{n}\otimes y_{n}$ in $ X\hat{\otimes}_{\gamma_{p}}Y$  converges to $u$. 
\end{pro}
\begin{proof}
	Let $u\in X\hat{\otimes}_{\gamma_{p}}Y$. Then for every $\delta>0$, there exist a sequence $(u_{i})_{i}$ such that
	\begin{equation}
		\sum_{i}\gamma_{p}(u_{i})\leq (1+\delta)\gamma_{p}(u)
	\end{equation}
and the series $\sum\limits_{i}u_{i}$ converges to $u$ in $ X\hat{\otimes}_{\gamma_{p}}Y$. We can obtain the sequence $(u_{i})_{i} \subset X\otimes Y$ such that $u=\sum\limits_{i=1}^{\infty}u_{i}$ with $\gamma_{p}(u_{1})<\gamma_{p}(u)+\delta$ and $\gamma_{p}(u_{i})<\frac{\delta^{2}}{4^{i}}$ for $i\geq 2$.
	Choose  a representation of $u_{i} = \sum\limits_{j=1}^{r_{i}}x_{j}\otimes y_{j}$ satisfying $\|(x_{1j})_{j}\|_{p^{*}}^{mid}<(\gamma_{p}+\delta)$ and $\|(y_{1j})_{j}\|_{p}\leq 1$ and for $i\geq 2$,
	$\|(x_{ij})_{j}\|_{p^{*}}^{mid}<\delta/2^{i}$ and $\|(y_{ij})_{j}\|_{p}<\delta/2^{i}$. Let  $(x_{n})_{n}$ and $(y_{n})_{n}$ be the infinite sequences obtained from concatenating the finite sequences $(x_{ij})_{j=1}^{r_{i}}$ and $(y_{ij})_{j=1}^{r_{i}}.$
\\\\	 For proving that the series $\sum\limits_{n=1}^{\infty}x_{n}\otimes y_{n}$ converges to $u$ in $X\hat{\otimes}_{\gamma_{p}}Y$ we need to show that  $(y_{n})_{n}\in \ell_{p}(Y)$ and $(x_{n})_{n}\in \ell_{p^{*}}^{mid}(X)$.
Now,
	\begin{equation}
		\|(y_{n})_{n}\|_{p}=\left(\sum_{i=1}^{\infty}\sum_{j=1}^{r_{i}}\|y_{ij}\|^{p}\right)^{1/p}<\left(1+\delta^{p}\sum_{i=2}^{\infty}2^{-pi}\right)^{1/p}.
	\end{equation} 
Let $(f_{m})_{m}\in B_{\ell_{p^{*}}^{w}(X^{*})}$ be arbitrarily chosen. Then using the definition of $\|.\|_{p^{*}}^{mid}$-norm, we have
\begin{align*}
	\left(\sum_{m=1}^{\infty}\sum_{n=1}^{\infty}|f_{m}(x_{n})|^{p^{*}}\right)^{1/p^{*}}=\left(\sum_{m=1}^{\infty}\sum_{i=1}^{\infty}\sum_{j=1}^{r_{i}}|f_{m}(x_{ij})|^{p^{*}}\right)^{1/p^{*}} \leq \left(\sum_{i=1}^{\infty}\left(\|(x_{ij})_{j=1}^{r_{i}}\|_{p^{*}}^{mid}\right)^{p^{*}}\right)^{1/p^{*}}.
\end{align*}
Taking supremum over the closed unit ball of $\ell_{p^{*}}^{w}(X^{*})$, we obtain
	\begin{equation}
	\|(x_{n})_{n}\|_{p^{*}}^{mid}\leq	\left((\gamma_{p}(u)+\delta)^{p^{*}}+\delta^{p^{*}}\sum_{i=2}^{\infty}2^{-p^{*}i}\right)^{1/p^{*}}.
	\end{equation}
Consequently, the series $\sum\limits_{n}x_{n}\otimes y_{n}$ converges in $X\hat{\otimes}_{\gamma_{p}} Y$ by using Lemma \ref{Convergence}. This completes the proof.
\end{proof}

\begin{thm}\label{gammap and Pipmid}
	For $1\leq p\leq\infty$, a linear operator $T:X\rightarrow Y^{*}$ is absolutely mid $p^{*}$-summing if and  only if the linear functional $\Phi_{T}$ corresponding to $T$  belongs to $(X\hat{\otimes}_{\gamma_{p}} Y)^{*}$. In this case, the operator norm of $\Phi_{T}$ coincides with $\pi_{p^{*}}^{mid}(T)$.	
\end{thm}
\begin{proof}
We prove the result for the case $1< p\leq \infty$ and   for $p=1$, it is clear from Proposition \ref{dp and gammap} that $\gamma_{1}=\pi$ and hence $\Pi_{1}^{mid}(X,Y^{*})=\mathcal{L}(X,Y^{*})$.	Consider $T\in\Pi_{p^{*}}^{mid}(X,Y^{*})$ and  $u=\sum\limits_{j=1}^{n}x_{j}\otimes y_{j} \in X\otimes Y$. Then
	\begin{align}\label{eq:33}
		\left|\Phi_{T}(u)\right|&=\left|\sum_{j=1}^{n}
		\left<y_{j},Tx_{j}\right>\right|\nonumber\\ 
		&\leq \left\|(y_{j})_{j=1}^{n}\right\|_{p}\left\|(Tx_{j})_{j=1}^{n}\right\|_{p^{*}}\nonumber\\ \
		&\leq \pi_{p^*}^{mid}(T) \left\|(x_{j})_{j=1}^{n}\right\|_{p^{*}}^{mid}\left\|(y_{j})_{j=1}^{n}\right\|_{p} \nonumber\\ 
		& \leq \pi_{p^*}^{mid}(T) \gamma_{p}(u)
	\end{align}
	by taking infimum over  all representations of $u$. Thus  $\Phi_{T} \in (X\otimes_{\gamma_{p}}Y)^{*}$.
	\\\\ Conversely, assume that $T:X\rightarrow Y^*$ be a continuous linear operator such that $\Phi_{T} \in (X\otimes_{\gamma_{p}}Y)^{*}$.
	Let $x_{1},x_{2},...,x_{n}\in X$. For each $j=1,2,\dots,n$, choose $y_{j}\in  Y$ and $\epsilon >0$ such that $\left<y_{j},Tx_{j}\right>=\|Tx_{j}\|^{p^{*}}$ and $\|y_{j}\|\leq(1+\epsilon)\|Tx_{j}\|^{p^{*}-1}$. Then
	\begin{align}\label{eq:36}
		\sum_{j=1}^{n}\|Tx_{j}\|^{p^{*}}&= \left|\sum_{j=1}^{n}\left<y_{j},Tx_{j}\right>\right|=	\left|\Phi_T(\sum_{j=1}^{n}x_{j}\otimes y_{j})\right|\\ \nonumber
		&\leq \|\Phi_{T}\| \left\|(x_{j})_{j=1}^{n}\right\|_{p^{*}}^{mid}\left\|(y_{j})_{j=1}^{n}\right\|_{p}\\ \nonumber
		&\leq \|\Phi_{T}\| \left\|(x_{j})_{j=1}^{n}\right\|_{p^{*}}^{mid}(1+\epsilon)\left(\sum_{j=1}^{n}\|Tx_{j}\|^{p(p^{*}-1)}\right)^{1/p}\nonumber.
	\end{align}
	Since $p(p^{*}-1)=p^{*}$, we have
	\begin{equation}\label{eq:35}
		\left(	\sum_{j=1}^{n}\|Tx_{j}\|^{p^{*}}\right)^{1/p^{*}}\leq \|\Phi_{T}\|(1+\epsilon)\left\|(x_{j})_{j=1}^{n}\right\|_{p^{*}}^{mid}.
	\end{equation}
	Thus $T\in \Pi_{p^{*}}^{mid}(X,Y^{*})$. Together, equations \eqref{eq:33} and \eqref{eq:35} gives
	 $\|\Phi_{T}\|=\pi_{p^{*}}^{mid}(T)$.
\end{proof}
As a consequence of Theorem \ref{gammap and Pipmid} and  the representation theorem for maximal operator ideals, we have
\begin{cor}\label{Maximal Pipmid}
	For $1\leq p\leq \infty$, $[\Pi_{p}^{mid},\pi_{p}^{mid}]$ is a maximal Banach operator ideal.
\end{cor}

\begin{cor}\label{T and T**}
Let $X,Y$ be Banach spaces and $1\leq p\leq \infty$. Then $T \in \Pi_{p}^{mid}(X,Y)$ if and only if $T ^{**}\in \Pi_{p}^{mid}(X^{**},Y^{**})$. Moreover, $\pi_p^{mid}(T) =\pi_p^{mid}(T^{**}).$
\end{cor}
\begin{proof}
	The result follows directly from Theorem \ref{gammap and Pipmid} and  \cite[\S 17.8  Corollary 4 ]{Defant}.
\end{proof}
The next proposition gives some equivalent descriptions for absolutely mid $p$-summing operators.
\begin{pro}
	Let $1\leq p\leq\infty$, $X,Y$ be Banach spaces and $T$ be an operator from $X$ to $Y$. Then the following are equivalent:
	\begin{enumerate}[(a)]
		\item There exists $C>0$ such that,
		\begin{equation}\label{eq:37}
			\left|\sum_{j=1}^{n}\left<Tx_{j},f_{j}\right>\right|\leq C \|(x_{j})_{j=1}^{n}\|_{p^{*}}^{mid}\|(f_{j})_{j=1}^{n}\|_{p}
		\end{equation}
		for every finite sequences $\left(x_{1},x_{2},...,x_{n}\right)$ and $\left(f_{1},f_{2},...,f_{n}\right)$ in $X$ and $Y^{*}$ respectively.
		\item There exists $C>0$ such that,
		\begin{equation}\label{eq:38}
			\|(Tx_{j})_{j=1}^{n}\|_{p^{*}}\leq C \|(x_{j})_{j=1}^{n}\|_{p^{*}}
		\end{equation}
		for every finite sequence $\left(x_{1},x_{2}\,...,x_{n}\right)$
		in $X$.
		\item The map $\hat{T}:\ell_{p^{*}}^{mid}(X)\rightarrow \ell_{p^{*}}(Y)$ is well defined and continuous.
		\item There exists $C>0$ such that,
		\begin{equation}\label{eq:39}
			\left(\sum_{j=1}^{n}\left\|(Tx_{j})\right\|^{p^{*}}\right)^{1/p^{*}}\leq C \|(x_{j})_{j=1}^{n}\|_{p^{*}}^{mid} 
		\end{equation}
		for every finite sequence  $\left(x_{1},x_{2}\,...,x_{n}\right)$
		in $X$. 
		\item There exists $C>0$ such that,
		\begin{equation}\label{eq:40}
			\left(\sum_{j=1}^{\infty}\left\|(Tx_{j})\right\|^{p^{*}}\right)^{1/p^{*}}\leq C \|(x_{j})_{j=1}^{n}\|_{p^{*}}^{mid} 
		\end{equation}
		for every sequences $\left(x_{j}\right)_{j}$ in $\ell_{p^{*}}^{mid}(X)$. 
	\end{enumerate}
	Furthermore,
$
		\pi_{p^{*}}(T)=\|\hat{T}\|=\inf\{C:\eqref{eq:37} \textnormal{ holds}\}=inf\{C:\eqref{eq:38}  \textnormal{ holds} \}= \inf\{C:\eqref{eq:39} \textnormal{ holds}\}= \inf\{C:\eqref{eq:40} \textnormal{ holds}\}.
$
\end{pro}
The next theorem characterizes the adjoints of absolutely mid $p$-summing operators. 
\begin{thm}\label{Adjoints}

	Let $X,Y$ be Banach spaces and $T:X\rightarrow Y$ be a continuous linear operator. Then for $1\leq p<\infty$,
	\begin{enumerate}[(i)]
	\item   $\hat{T^{*}}:\ell_{p^{*}}(Y^{*})\rightarrow (\ell_{p}^{mid})^{dual}(X^{*})$
	is well defined and continuous if and only if $T\in \Pi_{p}^{mid}(X,Y)$. In this case, $\pi_{p}^{mid}(T)=\|\hat{T^{*}}\|.$
	\item  $\hat{T}:\ell_{p^{*}}\left(X\right)\rightarrow(\ell_{p}^{mid})^{dual}\left(Y\right)$ is well defined and continuous if and only if $T^{*}\in \Pi_{p}^{mid}(Y^{*},X^{*})$. Moreover, $\pi_{p}^{mid}(T^{*})=\|\hat{T}\|$.
\end{enumerate}
	
\end{thm}
\begin{proof}
	\begin{enumerate}
		(i) 
	Let $T\in \Pi_{p}^{mid}(X,Y)$.  Then $T^{**}\in \Pi_{p}^{mid}(X^{**},Y^{**})$ by Corollary \ref{T and T**} and hence by Holder's inequality, we have  	\begin{align}\label{eq:3.19}
		\left\|\left(T^{*}(f_{j})_{j=1}^{n}\right)\right\|_{p,{mid}}^{dual}& =\sup_{(\psi_{j})_{j}\in B_{\ell_{p}^{mid}(X^{**})}}\sum_{j=1}^{n}\left|\left<\psi_{j},T^{*}(f_{j})\right>\right|\nonumber\\
		&= \sup_{(\psi_{j})_{j}\in B_{\ell_{p}^{mid}(X^{**})}}\sum_{j=1}^{n}\left|\left<T^{**}\psi_{j},(f_{j})\right>\right|\nonumber\\
		&\leq \sup_{(\psi_{j})_{j}\in B_{\ell_{p}^{mid}(X^{**})}}\left\|\left(T^{**}(\psi_{j})\right)_{j=1}^{n}\right\|_{p}\left\|(f_{j})_{j=1}^{n}\right\|_{p^{*}}\nonumber\\
		& \leq \sup_{(\psi_{j})_{j}\in B_{\ell_{p}^{mid}(X^{**})}} \pi_{p}^{mid}(T^{**})\left\|(\psi_{j})_{j=1}^{n}\right\|_{p}^{mid} \|(f_{j})_{j=1}^{n}\|_{p^{*}}\nonumber\\
		&\leq \pi_{p}^{mid}(T)\|(f_{j})_{j=1}^{n}\|_{p^{*}}
	\end{align}
	for every finite sequence $(f_{1},f_{2},...,f_{n})\in Y^{*}$. Thus $\hat{T^{*}}:\ell_{p^*}(Y^{*})\rightarrow (\ell_{p}^{mid})^{dual}(X^{*})$ is well defined and continuous.
	
	Conversly, let $\hat{T^{*}}:\ell_{p^*}(Y^{*})\rightarrow (\ell_{p}^{mid})^{dual}(X^{*})$  be well defined and continuous. Then by using $\ell_{p}^{dual}(Y^{*})\cong \ell_{p^{*}}(Y^{*}),$ we get
	\begin{align}\label{eq:3.20}
		\left\|(Tx_{j})_{j=1}^{n}\right\|_{p} & = \sup_{(f_{j})_{j}\in B_{\ell_{p^{*}}(Y^{*})}}\sum_{j=1}^{n}\left|\left<f_{j},Tx_{j}\right>\right|\nonumber\\
		&=\sup_{(f_{j})_{j} \in B_{\ell_{p^{*}}(Y^{*})}}\sum_{j=1}^{n}\left|\left<T^{*}f_{j},J_{X}x_{j}\right>\right|\nonumber\\
		&\leq \sup_{(f_{j})_{j} \in B_{\ell_{p^{*}}(Y^{*})}} C \left\|(J_{X}(x_{j}))_{j=1} ^{n}\right\|_{p}
		^{mid} \|(f_{j})_{j=1}^{n}\|_{p^{*}}\nonumber\\
		& \leq \|\hat{T^{*}}\| \left\|(x_{j})_{j=1} ^{n}\right\|_{p}
		^{mid}
	\end{align} 
	for each finite sequence $(x_{1},x_{2},...,x_{n})\subseteq X^{*}$, where $J_{X}:X\rightarrow X^{**}$ is the canonical embedding. Therefore, $T\in \Pi_{p}^{mid}(X,Y)$ and $$\pi_{p}^{mid}(T)=\|\hat{T^{*}}\|$$ by using \eqref{eq:3.19} and \eqref{eq:3.20}.
\item
Assume $T:X\rightarrow Y$ be such that $\hat{T}:\ell_{p^{*}}\left(X\right)\rightarrow(\ell_{p}^{mid})^{dual}\left(Y\right)$ is well defined and continuous. Now, define $u=\sum\limits_{j=1}^{n}f_{j}\otimes x_{j}$ in $ Y^{*}\hat{\otimes}_{\gamma_{p^{*}}}X$. Then
	\begin{align}\label{eq:42}
		\left|\Phi_{T^{*}}\left(u\right)\right|&= \left|\sum_{j=1}^{n}\left<T^{*}f_{j},x_{j}\right>\right|\nonumber\\ 
		&=\left|\sum_{j=1}^{n}\left<f_{j},Tx_{j}\right>\right|\nonumber\\ 
		&\leq \|\hat{T}\|\|\left(f_{j}\right)_{j}\|_{p}^{mid} \|\left(x_{j}\right)_{j}\|_{p^{*}}
	\end{align}
	by using the hypothesis. Thus $\Phi_{T^{*}}\in (Y^{*}\hat{\otimes}_{\gamma_{p^{*}}}X)^{*}$ and hence it follows from Theorem \ref{gammap and Pipmid}  that $T^{*}\in \Pi_{p}^{mid}(Y^{*},X^{*})$. 
	
	On the other hand, let $T^{*}$ be an absolutely mid  $p$-summing linear operator or equivalently,  $\Phi_{T^{*}}\in (Y^{*}\hat{\otimes}_{\gamma_{p^{*}}}X)^{*}$. Then for each $(x_{1},x_{2},...,x_{n})\subseteq X$ and  $(f_{1},f_{2},...,f_{n})\subseteq Y^{*},$ we have
	\begin{align}\label{eq:3.23}
		\sum_{j=1}^{n}\left|f_{j}(Tx_{j})\right|&=\left|\sum_{j=1}^{n}f_{j}(Tx_{j})\right|=\left|\Phi_{T^{*}}\left(\sum_{j=1}^{n}f_{j}\otimes x_{j}\right)\right|\nonumber\\
		&\leq \|\Phi_{T^{*}}\| \|\left(f_{j}\right)_{j}\|_{p}^{mid} \|\left(x_{j}\right)_{j}\|_{p^{*}}
	\end{align}
\end{enumerate}
	which proves that $\hat{T}:\ell_{p^{*}}\left(X\right)\rightarrow(\ell_{p}^{mid})^{dual}\left(Y\right)$ is well defined and continuous.  Also, together  \eqref{eq:42}  and \eqref{eq:3.23} imply $$\|\hat{T}\|=\pi_{p}^{mid}(T^{*}).$$
			\end{proof}
Finally we can identify the dual of $\Pi_{p}^{mid}$ as,
	\begin{cor}
		 Let $X$, $Y$ be Banach spaces and $1\leq p<\infty$. Then the component $(\Pi_{p}^{mid})^{dual}(X,Y)$  of the dual ideal  $(\Pi_{p}^{mid})^{dual}$ is given by,
		 \begin{equation*}
		  (\Pi_{p}^{mid})^{^{dual}}(X,Y)=\{T\in \mathcal{L}(X,Y): \hat{T}: \ell_{p^{*}}(X)\rightarrow (\ell_{p}^{mid})^{^{dual}}\left(Y\right) \text{is well defined and continuous}\}.
		  \end{equation*}
Furthermore, $(\pi_{p}^{mid})^{dual}(T)=\|\hat{T}\|$ for all $T\in (\Pi_{p}^{mid})^{dual}(X,Y)$.		 
	\end{cor}
\textbf{ACKNOWLEDGEMENT}. Financial support from the Department of Science Technology, India is acknowledged.

\end{document}